\documentclass[12pt]{article}
\usepackage{tikz,geometry, amssymb,amsmath,amsthm,microtype, breakurl}
\usepackage[colorlinks=true, allcolors=blue]{hyperref}
\author{Johan Andersson\thanks{Email:johan.andersson@oru.se \, Address: Department of Mathematics, School of Science and Technology, {\"O}rebro University, {\"O}rebro, SE-701 82 Sweden. } \and Linnea Rousu\thanks{Email:linnea.rousu@hotmail.com}}
 
\title{Polynomial approximation avoiding values in countable sets}
\theoremstyle{plain} 
\newtheorem{thm}{Theorem}  
\newtheorem{lem}{Lemma}
\newtheorem{prop}{Proposition}
\newtheorem{cor}{Corollary} 
\theoremstyle{definition} 
\newtheorem{prob}{Problem}
\newtheorem{conj}{Conjecture}
\newtheorem{ques}{Question}
\date{}

\def\cprime{$'$}

\newcommand{\C}{{\mathbb C}} 
\newcommand{\Q}{{\mathbb Q}}
\newcommand{\abs}[1]{{\left| {#1} \right|}}
\renewcommand{\emptyset}{\varnothing}
\newcommand{\eps}{\varepsilon}

\begin{document}

\maketitle  
      
\begin{abstract} 
We generalize a version of Lavrent\'ev's theorem which says that a function that is continuous on a compact set $K$ with connected complement and without interior points can be uniformly approximated as closely as desired by a polynomial without zeros on the set $K$, so that the polynomial can avoid values from any given countable set. We also prove a corresponding version of Mergelyan's theorem when the interior of $K$ is a finite union of Jordan domains, pairwise separated by a positive distance.
\end{abstract}
   
\maketitle
\section{Introduction}
Motivated by the Voronin universality theorem  for the Riemann zeta-function, the first author  \cite{Andersson}, \cite{Andersson2} generalized    the Lavrent\'ev theorem \cite{Lavrentiev} and the Mergelyan theorem \cite{Mergelyan}  so that in certain cases the approximating polynomial may be assumed to be zero-free on the compact set $K$. In particular the following conjecture\footnote{Gauthier (unpublished) also considered this problem in the seventies, see \cite[Remark 1]{Andersson2}.} \cite[Conjecture 2]{Andersson2} was stated
\begin{conj} \label{conj1}
  Assume that $K$ is a compact set with connected complement and that $f$ is a continuous function on $K$ which is analytic in the interior of $K$, such that $f$ is zero-free on  $K^\circ$.  Then given any $\varepsilon>0$ there exists a polynomial $p$ which is zero-free on $K$ such that
  \begin{gather} \label{ii}
    \max_{z \in K} \abs{f(z)-p(z)}<\varepsilon.
   \end{gather}
\end{conj}

While the conjecture in general seems quite difficult, it was proved in \cite{Andersson}  that the conjecture is true if the interior of $K$ is empty, and more generally in \cite[Theorem 6]{Andersson2} that the conjecture is true if the interior of $K$ is a finite union of Jordan domains, pairwise separated by a positive distance. The conjecture has since then been proved in increasing generality. Gauthier-Knese \cite{GauthierKnese} proved that the conjecture is true for ``chains of Jordan domains''. Khruschev \cite{Khru} proved the conjecture to be true if $K$ is locally connected. Andersson-Gauthier \cite{AnderssonGauthier} gave an independent proof\footnote{Actually a somewhat more general statement was proved, see e.g. \cite[Theorem 4]{AnderssonGauthier}.} of the ``trees of Jordan domains'' case. Furthermore it was proved by Danielyan \cite[Theorem 1]{Dan} that the conjecture is true for some $f$ with any given zero-set $Z \subset \partial K$. However he did not prove it for all functions with such a  zero-set $Z$ and the conjecture is thus still open for more complicated sets like the Cornucopia set\footnote{Example suggested by A. G. O-Farrell. See discussion in \cite[Section 4]{Andersson2}.}.

\section{Main results}

This paper deals with the generalization\footnote{Parts of the results in this paper  are included in the undergraduate thesis \cite{Rousu} written by the second author and supervised by the first author of the present paper.} of the zero-free approximation problem  to approximation by polynomials which avoids any countable set $A$ on the set $K$. Our main result is the following theorem. 
 \begin{thm} \label{TH1}
  Let $A \subset \C$ be any countable set and let $K \subset \C$ be  a compact set with connected complement, such that its interior  $K^\circ$ is the union of a finite number of Jordan domains, pairwise separated by a positive distance. Let $f$ be a continuous function on $K$ which is analytic in its interior $K^\circ$ such that $f(z) \not \in A$  if $z \in K^\circ$. Then  given any $\varepsilon>0$ there exists some polynomial $p$ such that $p(z) \not \in A$ if $z \in K$   and such that
  \[
    \max_{z \in K} \abs{f(z)-p(z)} <\varepsilon.
  \]
 \end{thm}
    In the proof of Theorem \ref{TH1} (see section \ref{sec4}), we  use similar arguments as in the proof of the corresponding result for the zero-free case \cite[Theorem 6]{Andersson2}. In particular  we use rescaling of Riemann mappings and the Carath\'{e}odory theorem. However,  to treat the parts of the set $K$ that are not in $\overline{K^\circ}$ we need some more complicated argument,   see Lemma \ref{LE2} in section \ref{sec3}. With Conjecture \ref{conj1} in mind we may ask if the corresponding result holds in the finite or countable case.
    \begin{ques} \label{q1}
       Does Conjecture \ref{conj1} hold \begin{enumerate}\itemsep0em  \item  if we replace ``zero-free'' with  ``avoid any set $A$ with two elements''? \item  if we replace ``zero-free'' with  ``avoid any finite set $A$''? \item if we replace ``zero-free'' with  ``avoid any countable set $A$''?
       \end{enumerate} 
     \end{ques}
     It is clear that the truth of Conjecture \ref{conj1} does not change if we replace ``zero-free'' with ''avoid any set $A$ with one element'', since we may always consider a shifted function (see proof of Lemma \ref{LE1} in section \ref{sec3}). While we have managed to answer Question \ref{q1} in the affirmative in the case when the interior of $K$ is the union of finitely many separated Jordan domains, by using the proof method of  \cite{Andersson2}, it is not clear to us how the methods of Gauthier-Knese \cite{GauthierKnese}, Andersson-Gauthier \cite{AnderssonGauthier} or Khruschev \cite{Khru} would generalize to this problem. In fact we do not even know how to treat the case when $A$ is a set with two elements and $K$ is the union of two closed discs that intersects at one point. We suggest this as an open problem.
     \begin{prob} Let $K=\{z:|z+1| \leq 1\} \cup \{z:|z-1| \leq 1\}$. Prove or give a counterexample:  For any  $\varepsilon>0$ and continuous function $f$ on $K$ that is analytic on $K^\circ$ such that $f(z) \not \in \{0,1\}$ if $z \in K^\circ$ there exists some polynomial $p$ such that \begin{gather*} \max_{z \in K} \abs{f(z)-p(z)}<\varepsilon,\end{gather*} and such that $p(z) \not \in \{0,1\}$ if $z \in K$.
        \end{prob}
  In the same way that the Lavrent\'ev  theorem \cite{Lavrentiev} is a direct consequence of the Mergelyan theorem \cite{Mergelyan}, for compact sets $K$ with empty interior, Theorem \ref{TH1} gives us the following version\footnote{See \cite[Sats 3.2.2]{Rousu}. This result also follows easily by using  Lavrent\'ev's theorem  to approximate $f$ by a polynomial $q$ and then using Lemma \ref{LE2} to approximate the polynomial $q$ by the polynomial $p$.}   of Lavrent\'ev's theorem  which generalizes the zero-free version  \cite[Theorem 1.1]{Andersson}. 
 \begin{thm} \label{TH2}
  Let $A \subset \C $ be any countable set, let $K \subset \C$ be  a compact set with connected complement and without interior points, and let $f$ be a continuous function on $K$. Then  given any $\varepsilon>0$ there exists some polynomial $p$ such that $p(z) \not \in A$ if $z \in K$ and such that
  \[
    \max_{z \in K} \abs{f(z)-p(z)}<\varepsilon.
  \]
 \end{thm}
 
 This result
 is stronger when $K$ is a larger set. Examples of large sets  (in terms of area measure) that satisfies the conditions of the theorem are $K_1=S+iS$, $K_2=[0,1]+iS$ and $K_3=[0,1] e^{2 \pi i S}$, where $S \subset [0,1]$ is a fat Cantor set. In fact we can choose $S$ to have one-dimensional measure arbitrarily close to $1$ so that $K_j \subset [0,1]+i[0,1]$ for $j=1,2$ have area measure arbitrarily close to $1$. The result is also stronger if $A$ is a larger set, such as a dense set in $\C$. An example of a dense set in $\C$ is when $A=\Q+\Q i$ is the set of rational complex numbers. Even stronger, if we use $A$ as the set of (complex) algebraic numbers in Theorem \ref{TH2} we obtain the following result on approximation by a polynomial that only takes transcendental values on the set $K$.

 \begin{cor} \label{COR1}
Let $K \subset \C$ be  a compact set with connected complement and without interior points, and let $f$ be a continuous function on $K$. Then  given any $\varepsilon>0$ there exists some polynomial $p$ such that
  $$
    \max_{z \in K} \abs{f(z)-p(z)}<\varepsilon,
  $$
  and such that  $p(z)$ is transcendental if $z \in K$.
 \end{cor}
 While our results hold for a countable set $A$ it would be interesting to know whether they hold for some uncountable set $A$. We do not consider this problem here, but  the following result gives some restriction on how we can choose the sets $K$ and $A$ if we want our approximation results to hold.
 
 \begin{prop} 
   If we remove the condition that $A$ is countable in  Theorem \ref{TH2}  then the  conclusion of Theorem \ref{TH2} is false if $K$ has some non trivial path-connected component  and $A$ has some non trivial connected component.
 \end{prop} 

 \begin{proof}
     We  give a proof by contradiction, by  assuming that the conclusion of  Theorem \ref{TH2} holds for this choice of $A$ and $K$. Since $K$ has some non trivial path-connected component and by using the fact that a path-connected set in the complex plane is arc-connected\footnote{\url{http://mathworld.wolfram.com/Arcwise-Connected.html}} we can find two points $z_1,z_2 \in K$ such that $z_1\neq z_2$, and some simple curve (Jordan arc) $B \subseteq  K$ with parametrization $\beta:[0,1] \to \C$ such that $\beta(0)=z_1$, $\beta(1)=z_2$ and $\beta([0,1])=B$. 
       Now let $a_1,a_2 \in A$ be such that $a_1 \neq a_2$ and such that $a_1$ and $a_2$ are contained in the same connected component of $A$. Let us consider the smooth curve $\Gamma$  with parametrization $\gamma:[0,1] \to \C$ that is given in the figure below,
       \vskip 6pt
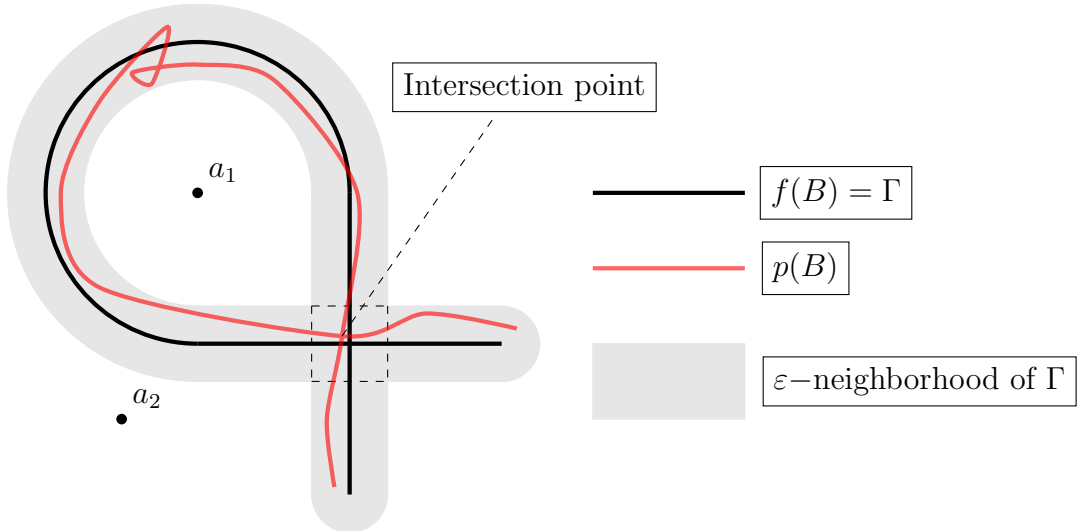
      \begin{tikzpicture}
     \filldraw[gray!20,fill=gray!20] (0:1.5) arc(0:270:1.5)--(270:2.5) arc(270:0:2.5)--cycle;
     \fill[gray!20]  (2.507,-4.0)  -- (-0.4,0) arc (-180:0: 0mm) 
                        -- (2.507,-4.0)    arc (0:-180:5.07mm);
                        \fill[gray!20]  (4,-2.51)  -- (-0.4,0) arc (-90:90: 0mm) 
                        -- (4,-2.51)    arc (-90:90:5.07mm);
      \filldraw[gray!20,fill=gray!20] (1.5,0) rectangle (2.5,-4); 
\filldraw[gray!20,fill=gray!20] (0,-1.5) rectangle (1.5,-2.5);
\filldraw[gray!20,fill=gray!20] (2.5,-1.5) rectangle (4,-2.5);
\draw[ultra thick](0:2) arc(0:270:2);
\draw[ultra thick](2,0)--(2,-4); 
\draw[ultra thick](0,-2)--(4,-2);
\draw[dashed](1.5,-1.5)--(2.5,-1.5);
\draw[dashed](1.5,-2.5)--(2.5,-2.5);
\draw[dashed](1.5,-1.5)--(1.5,-2.5);
\draw[dashed](2.5,-1.5)--(2.5,-2.5);
\draw [ultra thick, red, opacity=0.6] plot [smooth] coordinates {  (1.8,-3.9) (1.7,-3) (1.9,-1.9) (2.1,0) (1,1.5) (0,1.7) (-0.85,1.6)  (-0.6,1.45) (-0.4,2.2) (-1.3,1.2) (-1.8,0) (-1.3,-1.25) (1.9,-1.9)  (3,-1.6) (4.2,-1.8)  };
\draw[dashed](1.9,-1.9)--(3.9,1.0);
\fill (0,0)  circle[radius=2pt]; 
\fill (-1,-3)  circle[radius=2pt];
\foreach \Point/\PointLabel in {(0,0)/a_1, (-1,-3)/a_2}
\draw[fill=black] \Point circle (0.05) node[above right] {$\PointLabel$};
\node[draw] at (4.3,1.4) {Intersection point};
\draw[ultra thick](5.2,0)--(7.2,0);
\draw[red, ultra thick,opacity=0.6](5.2,-1)--(7.2,-1);
\filldraw[gray!20,fill=gray!20] (5.2,-3.0) rectangle (7.2,-2.0);
\node[draw] at (8.4,0) {$f(B)=\Gamma$};
\node[draw] at (8.0,-1) {$p(B)$};
\node[draw] at (9.5,-2.5) {$\varepsilon-$neighborhood of $\Gamma$};
\end{tikzpicture}
\vskip 6pt

   \noindent   and let the continuous function $f:K \to \C$ be defined by $f(z)=\gamma(\beta^{-1}(z))$ if $z \in B$, such that $f(B)=\Gamma$ and by the Tietze extension theorem \cite[Theorem 20.4]{Rudin}  as any continuous extension to $K$.     
      Let us now use our assumption that the conclusion of Theorem \ref{TH2} holds and construct a polynomial $p$ such that $p(z) \not \in A$ if  $z \in K$ and such that  
        \begin{gather*}
           \abs{f(z)-p(z)}<\varepsilon, \qquad z \in K,
         \end{gather*}
        holds.        If $\varepsilon>0$ is sufficiently small, then $p(B)$   contains a Jordan curve $J \subseteq p(B)$  surrounding $a_1$ but not $a_2$. This is a consequence of the fact that the curve $p(B)$  follows the $\varepsilon$-neighborhood of $\Gamma$ (grey in figure) and by the crossed arcs lemma\footnote{\url{http://www.cut-the-knot.org/blue/JCT/JCT\_Part4.shtml}} must intersect in the marked square  with side $2 \varepsilon$ (see figure) centered at the intersection point of the curve $\Gamma$.  It follows by the Jordan curve theorem that $J^\complement$ is  not connected and that $a_1$ and $a_2$ do not lie in the same connected component of $J^\complement$. Since $A \subseteq p(K)^\complement \subseteq p(B)^\complement  \subseteq J^\complement$ this contradicts our assumption that $a_1$ and $a_2$ lie in the same connected component of $A$.
    \end{proof}

\section{Some lemmas on polynomials approximating polynomials \label{sec3}}
In order to prove Theorem \ref{TH1} we need some useful lemmas on polynomial approximation. In fact we only need Lemma \ref{LE2}, but in order to prove Lemma \ref{LE2} we need the following Lemma. 
\begin{lem}
 \label{LE1}
  Let $p$ be a polynomial and let $K$ be a compact subset of $\C$. Then given any $\varepsilon>0$ and any complex number $a$ there exists some polynomial $q$  of the same degree as $p$
   such that
  $$
   \max_{z \in K} \abs{p(z)-q(z)}<\varepsilon,
  $$
  and such that $q(z) \neq a$ for $z \in \partial K$.
 \end{lem}
The following proof is a shifted variant of  the proof of  \cite[Theorem 1.1]{Andersson}.
 \begin{proof}
Let $g(z)=p(z)-a$ be such that \[g(z)=c_0\prod_{k=1}^m (z-z_k)\] where $z_k$ denotes the zeros of $g(z)$. 
Since $\partial K$ has no interior points there exist sequences $z_{k,n}\in (\partial K)^\complement$ such that $z_{k,n}\rightarrow z_k$. 
Let \begin{equation}
g_n(z)=c_0\prod_{k=1}^m (z-z_{k,n}). 
\end{equation}
Since $z_{k,n}\in (\partial K)^\complement$ we obtain $g_n(z)\neq 0$ for $z\in \partial K.$ 
Since the coefficients converge it is clear that $g_n(z)$ converges uniformly to $g(z)$ on $K$. Hence, there is some $n$ such that \begin{gather} \label{ineq} |g_n(z)-g(z)|<\varepsilon,\qquad z\in K. \end{gather}
Let $q(z)=g_n(z)+a.$ Since $g_n(z)\neq 0$ for $z\in \partial K$ it follows that $q(z) \neq a$ if $z \in \partial K$,
and since $p(z)-q(z)=g(z)-g_n(z)$ it follows from \eqref{ineq} that \[|p(z)-q(z)|<\varepsilon,\qquad z\in K.\]
\end{proof}

\begin{lem} \label{LE2}
  Let $q$ be a polynomial and let $K$ be a compact subset of $\C$. Then given any $\varepsilon>0$ and any countable set $A \subset \C$  there exists some polynomial $p$  such that 
  $$
   \max_{z \in K} \abs{q(z)-p(z)}<\eps,
  $$
  and such that $p(z) \not \in A$ for $z \in \partial K$.
 \end{lem}
 \begin{proof} 
 Let $A=\{a_1,a_2,\dots\}$ and let $m$ be the degree of $q$. Let   $0<\varepsilon_0 <\varepsilon$ and  $p_0(z):=q(z)$. 
    For $j=1, 2, \ldots$ there is, according to Lemma \ref{LE1}, some polynomial $p_j$ of degree $m$ such that  
\begin{equation}\label{neq_a}
     \delta_j:= \min_{z \in \partial K} \abs{p_j(z)-a_j}>0, 
    \end{equation}
    and such that \begin{equation} \label{appr}
|p_j(z)-p_{j-1}(z)|<\frac{\eps_{j-1}}{2},\qquad   z \in K,
\end{equation} 
where $\eps_j>0$ for $j\geq 1$ is definied recursively so that 
\begin{equation} \label{epsj}
    \eps_j < \min \left(\delta_j,\frac{\eps_{j-1}} 2 \right).
\end{equation} 
By  the inequalities \eqref{appr}, \eqref{epsj}, and the triangle inequality we find for $0 \leq k <l$ that
\begin{gather}\label{Cauchy}
\begin{split}
\abs{p_l(z)-p_k(z)}&\leq |p_l(z)-p_{l-1}(z)|+ \cdots +|p_{k+1}(z)-p_{k}(z)| 
\\
&<\frac{\eps_{l-1}}{2}
+\cdots+\frac{\eps_k}{2}<\sum_{j=1}^{l-k}\frac{\eps_k}{2^j}<\sum_{j=1}^{\infty}\frac{\eps_k}{2^j}=\eps_k,
\end{split} \qquad z \in K.
\end{gather} 
 By \eqref{epsj} and \eqref{Cauchy} it follows that  $\{p_j\}_{j=1}^\infty$ is a Cauchy-sequence in the vector space of polynomials of degree at most $m$ equipped with the sup-norm on $K$. Since this space is complete then \begin{equation}\label{p} p(z)=\lim_{j\to \infty} p_j(z) \end{equation}   is a polynomial of degree at most\footnote{It is clear by the construction that the polynomial will in fact have degree exactly $m$.} $m$. 
The inequality \eqref{Cauchy} yields \begin{equation}\label{pj-p}
    |p(z)-p_j(z)|\leq \eps_j,\qquad j\geq 0,~z\in K.
\end{equation}
For the polynomial $p$ definied by \eqref{p}, the inequalites \eqref{neq_a},  \eqref{epsj}, \eqref{pj-p}  and the triangle inequality gives us
\begin{equation} \label{yy}
    |p(z)-a_j|\geq |p_j(z)-a_j|-|p(z)-p_j(z)| \geq \delta_j-\eps_j>0,\qquad z\in \partial K,
\end{equation} for all $a_j\in A.$
  The conclusion of our lemma follows by \eqref{pj-p} and \eqref{yy}, by recalling that $p_0(z)=q(z)$ and $0<\varepsilon_0<\varepsilon$.
 \end{proof}

\section{Proof of Theorem \ref{TH1} \label{sec4} }

 Since $A \subseteq (K^\circ)^\complement$ we also have $\overline A \subseteq \overline{(K^\circ)^\complement}=(K^\circ)^\complement.$ Thus we have   that $f(z) \not \in \overline{A}$ if $z \in K^\circ$. Now let  $K^\circ=\cup_{j=1}^n O_j$ where $O_j$ are Jordan domains such that $\overline{O_j} \cap \overline{O_i}=\emptyset$ if $i \neq j$. 
 Let $D= \{z \in \C: |z|<1\}$  denote the open unit disc. By the Carath{\'e}odory theorem\footnote{Also called the Carath{\'e}odory-Osgood-Taylor theorem since it was proved independently by  Carath{\'e}odory  \cite{Cara} and Osgood-Taylor \cite{Osgood}.} \cite[Theorem 14.19]{Rudin}, the  Riemann mappings $\phi_j: O_j \to D$ extend to homeomorphisms $\phi_j:  {\overline{O_j}} \to \overline{D}.$
It is clear that
$$f(z)=f(\phi_j^{-1}(\phi_j(z))), \qquad z \in \overline{O_j}. $$
 Let us now define
$$g(z)=f(\phi_j^{-1}((1-\xi_j)\phi_j(z))), $$ whenever $z \in \overline{O_j}$
for sufficiently small $0<\xi_j<1$ such that
\begin{gather*}  |g(z)-f(z)| \leq  \frac \varepsilon 3, \qquad  z \in \overline{O_j}. \end{gather*} 
The Tietze extension theorem \cite[Theorem 20.4]{Rudin} allows $g$ to be extended  to a continuous function on $K$ such that
\begin{gather} \label{ut2}
\abs{g(z)-f(z)} \leq \frac \varepsilon 3, \qquad  z \in K.
\end{gather} 
 By the construction it is clear that $g$ is continuous on $K$,  that $g$ is analytic on $K^\circ$ and that $g(z) \not \in \overline{A}$ if $z \in \overline{K^\circ}$. Thus
   the compact set $g(\overline{K^\circ})$ and the closed set $\overline{A}$ are disjoint and must hence be separated by a positive distance $\delta$ so that  
    \begin{gather}
        \abs{g(z)-a} \geq \delta, \qquad z \in \overline{K^{\circ}},\,  a \in  \overline{A}. 
      \label{Aa}
     \end{gather}
 By Mergelyan's theorem (see \cite{Mergelyan} or \cite[Theorem 20.5]{Rudin}) 
 we can choose a polynomial $q$ such that
\begin{gather} \label{ut3}
|q(z)-g(z)| < \min (\varepsilon/3,\delta/2), \qquad  z \in K.
\end{gather}
 By Lemma \ref{LE2} there exists some polynomial $p$ such that
      \begin{gather} \label{B}
      \abs{p(z)-q(z)}<\min \left(\varepsilon/3,\delta/2 \right), \qquad  z \in K,
    \end{gather}
    and such that $p(z) \not \in A$ if $z \in \partial K$. By the inequalities \eqref{Aa}, \eqref{ut3}, \eqref{B} and the triangle inequality it follows that also $p(z) \not \in A$ if $z \in K^\circ$. Thus $p(z) \not \in A$ if $z \in K$. Finally it follows from the  inequalities \eqref{ut2}, \eqref{ut3}, \eqref{B} and the triangle inequality that 
   \begin{gather*}
    \max_{z \in K} \abs{f(z)-p(z)} <\varepsilon.
  \end{gather*}
   \qed

\bibliographystyle{plain}

\end{document}